%% file: maiin.tex
\documentclass[preprint]{imsart}

\usepackage{hyperref}
\hypersetup{colorlinks,citecolor=blue,urlcolor=blue}

\usepackage{mathptmx} 
\usepackage[OT1]{fontenc} 
\usepackage{mathtools}

\usepackage[english]{babel}

\usepackage{amsmath,amsthm,verbatim,amssymb,amsfonts,amscd, graphicx}
\usepackage{amssymb,bm}
\usepackage{mathrsfs}

\RequirePackage[numbers]{natbib}
\usepackage{constants}
\usepackage{enumerate}
\usepackage{algorithm}
\usepackage{algorithmic}
\usepackage{dsfont}
\usepackage{float}
\usepackage[mathlines]{lineno}

\newcommand{\be}[1]{\begin{equation*}#1\end{equation*}}
\newcommand{\ben}[1]{\begin{equation}#1\end{equation}}

\def\ml#1{\begin{multline*}{#1}\end{multline*}}
\def\mln#1{\begin{multline}{#1}\end{multline}}

\newcommand{\rom}[1]{\uppercase\expandafter{\romannumeral #1\relax}}
\newcommand{\beas}{\begin{eqnarray*}}
\newcommand{\enas}{\end{eqnarray*}}
\newcommand{\bea}{\begin{eqnarray}}
\newcommand{\ena}{\end{eqnarray}}
\newcommand{\bms}{\begin{multline*}}
\newcommand{\ems}{\end{multline*}}
\newcommand{\bels}{\begin{align*}}
\newcommand{\enls}{\end{align*}}
\newcommand{\bel}{\begin{align}}
\newcommand{\enl}{\end{align}}

\newcommand{\ignore}[1]{}

\def\blfootnote{\xdef\@thefnmark{}\@footnotetext}

\newcommand{\dotp}[2]{\left\langle#1,#2\right\rangle}
\newcommand{\m}{\mathcal}
\newcommand{\E}{\mathbf{E}}
\newcommand{\mb}{\mathbb}
\newcommand\argmin{\mathop{\mbox{argmin}}}

\newcommand{\card}{\mathrm{Card}}

\def\r{\right}
\def\l{\left}

\newcommand{\var}{\mbox{Var}}

\newcommand{\wh}{\widehat}
\newcommand{\wt}{\widetilde}

\newcommand{\Pt}{\mathcal{P}_{\tau,a}}

\newcommand{\Prob}{\mathrm{Pr}}

\newcommand{\vertiii}[1]{{\left\vert\kern-0.25ex\left\vert\kern-0.25ex\left\vert #1 
    \right\vert\kern-0.25ex\right\vert\kern-0.25ex\right\vert}}

\newtheoremstyle{mytheoremstyle}
  {\topsep} 
  {0pt} 
  {\itshape} 
  {} 
  {\bfseries} 
  {.} 
  {.5em} 
  {} 
\theoremstyle{mytheoremstyle}
\numberwithin{equation}{section}

\newtheorem{theorem}{Theorem}[section]

\newtheorem{proposition}{Proposition}[section]
\newtheorem{lemma}{Lemma}[section]
\newtheorem{assumption}{Assumption}
\newtheorem{remark}{Remark}[section]

\begin{document}
\nolinenumbers
\begin{frontmatter}
\title{
Robust and Tuning-Free Sparse Linear Regression via Square-Root Slope}


\begin{aug}
\author[A]{\fnms{Stanislav} \snm{Minsker,}\ead[label=e1]{minsker@usc.edu}}
\author[B]{\fnms{Mohamed} \snm{Ndaoud}\ead[label=e2]{ndaoud@essec.edu}}
\and
\author[A]{\fnms{Lang} \snm{Wang}\ead[label=e3]{langwang@usc.edu}}

\address[A]{Department of Mathematics,\\ University of Southern California\\
\printead{e1,e3}}

\address[B]{Department of Information Systems, Decision Sciences and Statistics\\
ESSEC Business School\\
95000 Cergy, France\\
\printead{e2}}


\end{aug}

\begin{abstract}
We consider the high-dimensional linear regression model and assume that a fraction of the measurements are altered by an adversary with complete knowledge of the data and the underlying distribution. We are interested in a scenario where dense additive noise is heavy-tailed while the measurement vectors follow a sub-Gaussian distribution. Within this framework, we establish minimax lower bounds for the performance of an arbitrary estimator that depend on the the fraction of corrupted observations as well as the tail behavior of the additive noise. Moreover, we design a modification of the so-called Square-Root Slope estimator with several desirable features: (a) it is provably robust to adversarial contamination, and satisfies performance guarantees in the form of sub-Gaussian deviation inequalities that match the lower error bounds, up to logarithmic factors; (b) it is fully adaptive with respect to the unknown sparsity level and the variance of the additive noise, and (c) it is computationally tractable as a solution of a convex optimization problem.
To analyze performance of the proposed estimator, we prove several properties of matrices with sub-Gaussian rows that may be of independent interest.
\end{abstract}


\tableofcontents
\end{frontmatter}

\mathtoolsset{showonlyrefs=true}

\section{Introduction}
Robust statistics, broadly speaking, is an arsenal of estimation and inference techniques that are resistant to model perturbations. Data generated from a perturbed model will often contain atypical observations, commonly referred to as outliers. This paper is devoted to robust estimation in the context of high-dimensional sparse linear regression. Assume that a sequence of random pairs 
$(X_1,y_1),\ldots,(X_n,y_n)$ is generated according to the model
\[
y_i = X_i^T\beta^* + \sqrt{n}\theta^*_i +\sigma \xi_i, \quad i=1,\ldots,n.
\]
Here, each $y_i\in \mb R$ is a linear measurements of an unknown vector $\beta^\ast \in \mb R^p$ that has $s$ non-zero coordinates. The measurement vectors $X_i\in \mb R^p, \ j=1,\ldots,n$ are independently sampled from a distribution with unknown covariance matrix $\Sigma$. We assume that the measurements $y_i$ are contaminated by the noise $\sigma \xi_i$ where $\sigma>0$ and $\xi_1,\ldots,\xi_n$ are i.i.d. random variables with unit variance, independent from $X_1,\ldots,X_n$. Finally, the adversarial noise is modeled by the additive term $\sqrt{n}\theta^\ast_i$\footnote{See remark \ref{remark:scaling} related to the $\sqrt{n}$ factor.}, where the sequence 
$\theta^\ast_1,\ldots,\theta^\ast_n$ has $o<n$ non-zero elements and is generated by an adversary who has access to $\{(y_i,X_i,\xi_i)\}_{i=1}^n, \beta^\ast$, $\sigma$, as well as the joint distribution of all random variables involved. 
We are interested in the situation when (a) when $p$ is possibly much larger than $n$ but $s$ is smaller than $n$, and (b) the random variables 
$\{\xi_i\}_{i=1}^n$ are possibly heavy-tailed, distributed according to a law with polynomially decaying tails. 

The well-known Lasso estimator \cite{tibshirani1996regression}, as well as its sibling, the Dantzig selector \cite{candes2007dantzig}, provably achieve strong performance guarantees in the sparse prediction and estimation tasks. For example, the Lasso estimator is the solution to the following optimization problem: 
\[
\widetilde\beta = \argmin_{\beta\in \mb R^p} \l[\frac{1}{n}\sum_{j=1}^n (y_j - X_j^T \beta)^2 + \lambda\|\beta\|_1\r]
\]
where $\lambda>0$ is the regularization parameter and $\|\cdot\|_p$ stands for the $\ell_p$ norm of a vector, $p\geq 1$. It is known that theoretically optimal value of the parameter $\lambda$ is proportional to $\sigma \sqrt{\frac{\log(ep/s)}{n}}$ \cite{bellec2018slope,ndaoud}, in particular, it depends on the unknown variance of the noise as well as the unknown sparsity level $s$. In addition, the Lasso estimator $\widetilde\beta$ is not robust to the presence of gross outliers, i.e. the norm $\|\beta^\ast-\widetilde\beta\|_2$ can be arbitrarily large if $\theta^\ast$ has just 1 non-zero element that can take arbitrary values.
In this paper, we propose robust version of the pivotal Slope (``Sorted $\ell$-One Penalized Estimation'') algorithm \cite{derumigny2018improved} that is provably robust to the heavy-tailed additive noise and the adversarial corruption; moreover, it is tuning-free. 
The proposed estimator combines the ideas behind the original Slope algorithm \cite{bogdan2015slope} that eliminates dependence of the optimal choice of $\lambda$ on the sparsity level $s$, the square-root Lasso \cite{belloni2011square} that allows to set $\lambda$ independently of the noise variance $\sigma^2$, and moreover takes advantage of the robustness stemming from connections between the Huber's loss and the $\ell_1$ - penalized squared loss \cite{sardy2001robust,gannaz2007robust}. Specifically, we prove that the estimator $\wh \beta$ produced by robust pivotal Slope and formally defined via \eqref{estimator:sqrt lasso} below admits the following performance guarantees under suitable assumptions on the covariance matrix $\Sigma$ of the design vectors: 
\begin{enumerate}
\item[(a)] If both the design vectors $X_j$ and the noise variables $\xi_j, \ j=1,\ldots,n$ have sub-Gaussian distributions, then
\[
\|\wh{\beta} - \beta^*\|^2_\Sigma  \lesssim \sigma^2 \l(\frac{s\log(ep/s) }{n} + \l(\frac{o\log(n/o)}{n} \r)^2 +\frac{\log(1/\delta)}{n} \r)
\]
with probability at least $1-\delta$, where $\lesssim$ denotes the inequality up to an absolute constant and $\|x\|^2_\Sigma := \dotp{\Sigma x}{x}$. In particular, the upper bound is minimax optimal with respect to sparsity level $s$ and the number of outliers $o$.
\item[(b)] If $X_j$'s are sub-Gaussian but $\xi_j$'s are heavy tailed, meaning that $\E(|\xi|^\tau) < \infty $ for some $\tau\geq 4$, and in the absence of adversarial contamination (i.e. $\theta^\ast \equiv 0$), 
\[
\|\wh{\beta} - \beta^*\|^2_\Sigma  \lesssim \sigma^2 \l(\frac{s\log(ep/s) }{n} +\frac{\log(1/\delta)}{n} \r)
\]
with probability at least $1-\delta$. In other words, $\|\wh{\beta} - \beta^*\|_\Sigma$ admits sub-Gaussian deviation guarantees despite the fact that the noise is allowed to be heavy-tailed.
 \item[(c)] Finally, if $\E(|\xi|^\tau) <\infty$ for some $\tau \geq 2$ and $s\log(p/s) + \log(1/\delta) + o \lesssim n$, a version of the proposed estimator satisfies
\ml{
\|\wh{\beta} - \beta^*\|^2_\Sigma  \lesssim \sigma^2 \l(\frac{s\log(ep/s) }{n} \r. 
\\
\l. + \l(\frac{o}{n}\r)^{2-2/\tau}\log\l( \frac{n}{o}\r) \l(1 + \l(\frac{o}{\log(1/\delta)}\r)^{2/\tau}\r) +\frac{\log(1/\delta)}{n} \r)
}
with probability at least $1-\delta$. It implies that whenever $\log(1/\delta)\gtrsim o$, the upper bound depends  optimally (up to a logarithmic factor) on the number of adversarial outliers, as well as on the sparsity level $s$. Note that the deviations guarantees are again sub-Gaussian. 
\end{enumerate}

\subsection{Structure of the paper.}

The rest of the exposition is organized as follows: notation and key definitions are summarized in section \ref{section:topic2_notation}. In section \ref{section:topic2_formulation}, we explain the main ideas leading to the definition of the pivotal Slope estimator and state the theoretical guarantees related to its performance, along with the information-theoretic lower bounds. This is followed by a discussion and comparison to existing results in section \ref{section:discussion}. Finally, the proofs of the main results are presented in the appendix.



\subsection{Notation.}
\label{section:topic2_notation}

Absolute constants that do not depend on any parameters of the problem are going to be denoted by $C,C',C_1$, etc as well as $c,c',c_1$, with the convention that capital $C$ stands for ``a sufficiently large absolute constant'' while the lower case $c$ is a synonym of ``a sufficiently small absolute constant''. It is assumed that $C$ and $c$ can denote different absolute constants in different parts of the expression. 
For $a,b\in \mb R$, let $a \vee b := \max \{a,b\}$ and $a \wedge b := \min \{a,b \}$.

Given a vector $v \in \mb{R}^{p}$, we denote its $\ell_1$ and $\ell_2$- norms via $\l\|v\r\|_1 := \sum_{i=1}^{p} |v_i|$ and $\l\|v\r\|_2 := \sqrt{ \sum_{i=1}^{p} |v_i|^2}$ respectively. If $\Sigma\in \mb R^{p\times p}$ is a symmetric positive-definite matrix, we define $\|v\|_\Sigma:=\dotp{\Sigma v}{v}^{1/2}$.  
Given two vectors $u \in \mb{R}^{n}$ and $v \in \mb{R}^{p}$, let $[u;v] \in \mb{R}^{
n} \times \mb{R}^{p}$ be the $(p + n)$-dimensional vector created by the vertical concatenation of $u$ and $v$. Let $\gamma_1\geq \gamma_2\geq\ldots\geq \gamma_p\geq 0$ be a non-increasing sequence. The corresponding sorted $\ell_1$ norm is defined as
\[
\l\|v\r\|_{\gamma} := \sum_{i=1}^p \gamma_i |v|_{(i)},
\]
where $|v|_{(i)}$ is the i-th largest coordinate of the vector $(|v|_1,\ldots,|v|_p)$; the fact that this is indeed a norm is established in \citep[][Proposition 1.2]{bogdan2015slope}.

Capital $S$ and $O$ will be reserved for the supports of vectors $\beta^*$ and $\theta^*$, the subsets of $\{1,\ldots,p\}$ and $\{1,\ldots,n\}$ respectively that contain the indices of non-zero coordinates of these vectors. We will also set $s = |S| := \card(S) $ and $o = |O| := \card(O)$. 


\section{Main results.}
\label{section:topic2_formulation}

Recall that we observe $n$ random pairs of predictor-response values $\big(X_1, y_1\big), \ldots, \big(X_n, y_n\big) \in \mb{R}^{p} \times \mb{R}$ that are assumed to be generated according to the model
\begin{equation*}
y_i = X_i^T\beta^* +\sqrt{n} \theta_i^* + \sigma \xi_i,\qquad i=1,\ldots,n.
\end{equation*}
Alternatively,
\begin{equation}
\label{model}
Y = X\beta^* + \sqrt{n}\theta^* + \sigma \xi,
\end{equation}
where $X = [X_1, \ldots, X_n]^T$ is the $n \times p$ design matrix, $Y = \l(y_1, \ldots, y_n \r)^T$ is the response vector, $\xi = \l( \xi_1, \ldots, \xi_n \r)^T$ is the additive noise vector and $\theta^* = \l( \theta^\ast_1, \ldots,\theta^\ast_n\r)^T$ is the vector of adversarial outliers. 
\begin{remark}
\label{remark:scaling}
The $\sqrt{n}$ factor in front of $\theta^\ast$ is introduced for technical convenience: with this scaling, the columns of the augmented design matrix $[X \, | \,\sqrt{n} \,I_n]$, where $I_n$ is $n\times n$ identity matrix, are of similar length.
\end{remark}
\noindent Let us now define the robust pivotal Slope estimator. To this end, set
\[
Q(\beta,\theta) := \frac{1}{2n}\sum_{j=1}^{n} (y_j-X_j^T{\beta} -\sqrt{n}\theta_j)^2 = \frac{1}{2n}\l\|Y-X\beta-\sqrt{n}\theta\r\|_2^2,
\]
and let 
\begin{equation}
\label{equation:sqrt lasso loss}
L(\beta,\theta) = Q(\beta,\theta)^{\frac{1}{2}} + 
\l\|\beta\r\|_{\lambda} + \l\|\theta\r\|_{\mu}
\end{equation}
for some positive non-increasing sequences $\{\lambda_j\}_{j=1}^p, \ \{\mu\}_{j=1}^p$. The pivotal Slope estimator $\wh\beta$ of $\beta^\ast$ is then defined via the solution of a convex minimization problem
\begin{equation}
\label{estimator:sqrt lasso}
(\wh{\beta},\wh{\theta}) = \argmin_{\beta\in\mb{R}^p,\theta\in\mb{R}^n} L(\beta,\theta).
\end{equation}
The estimator in \eqref{estimator:sqrt lasso} can be seen as a generalization of the square-root Slope estimator (see \cite{stucky2017sharp,derumigny2018improved}). 
The idea of introducing the square root of the quadratic term $\sqrt{Q(\beta,\theta)}$, as opposed to $Q(\beta,\theta)$ itself, was originally developed in \cite{belloni2011square} for the Lasso estimator with the goal of removing the dependence of the regularization parameters on unknown $\sigma$, while retaining the convexity of the loss function. Note that estimator \eqref{estimator:sqrt lasso} is equivalent to the $\argmin$ over $\beta \in \mb{R}^{p}$, $\theta \in \mb{R}^{n}$ and $\sigma > 0$ of the loss function
\begin{equation}
\label{estimator:scaled lasso}
\wt{L}(\beta,\theta,\sigma) = \frac{Q(\beta, \theta)}{\sigma} + \sigma + \l\|\beta\r\|_{\lambda} + \l\|\theta\r\|_{\mu},
\end{equation}
provided that the minimum is attained at a positive value of $\sigma$ (see \cite[Chapter 3]{van2016estimation}). Indeed, the term $Q(\beta,\theta)^{\frac{1}{2}}$ in \eqref{equation:sqrt lasso loss} appears when one performs minimization of $\wt{L}(\beta,\theta,\sigma)$ 
with respect to $\sigma > 0$ first, and the optimal value of $\sigma$ can be viewed as an estimator of the unknown standard deviation of the noise. 
The couple $(\wh\beta,\wh\theta)$ can in turn be viewed as the usual square root Slope estimator for the vector $[\beta^\ast;\theta^\ast]$ of unknown regression coefficients corresponding the augmented design matrix $\l[ X; I_n\r]$. This is a natural approach since both $\beta^\ast$ and $\theta^\ast$ are sparse vectors. 

In the following subsections, we will show that the estimator defined in \eqref{estimator:sqrt lasso} achieves the optimal error bound under suitable choices of the sequences $(\lambda)_p$ and $(\mu)_n$. To this end, we will need the following assumptions: 
\begin{assumption}\label{ass:1} $\xi_i$'s are i.i.d. random variables with distribution $P_\xi$ in $\Pt$ for some $\tau\geq 2$ with $\E(\xi_i)=0$ and $Var(\xi_i) = 1$, where
	\[
	\Pt = \{ P_\xi \text{ such that } \E(|\xi|^\tau)\leq a^\tau\}.
	\]
	When $\tau = \infty$, we instead impose the condition $\Prob(|\xi| \geq t) \leq 2\exp(-(t/a)^2)$. We will also assume that \emph{$\tau$ or its lower bound is known} and that $a$ is bounded by a sufficiently large numerical constant.
\end{assumption}
\begin{assumption}\label{ass:2} Assume that $\xi$ satisfies a ``small ball-type'' condition, namely $$\E(\xi^2_i \mathbf{1}\{|\xi_i|\leq 1/2\} ) \geq 1/4$$ (the specific choice of the constants $1/2$ and $1/4$ is not of essence).
\end{assumption}
\begin{assumption}\label{ass:3}
    For $i=1,\dots,n$ $X_i = \Sigma^{1/2}Y_i$ for some (unknown) matrix $\Sigma$ satisfying $\Sigma_{ii} \leq 1$ and $Y_1,\ldots,Y_n$ are i.i.d. centered $1$-sub-Gaussian random vectors such that $\E(Y_1Y_1^\top) = I_n$ . Here, ``$1$-sub-Gaussian'' means that $\mb E e^{\lambda \dotp{Y_1}{v}}\leq e^{\lambda^2\frac{\|v\|^2_2}{2}}$ for any $v\in \mb R^p$. 
\end{assumption}
\noindent We will also need to introduce the following objects:
\begin{itemize}
	\item Define the cone
	\begin{equation}
	\label{def:dimension reduction cone original}
	\mathcal{C}(s,c_0) = \l\{ u \in \mb{R}^p: \l\|u\r\|_\lambda  \leq c_0 \sqrt{\sum_{i=1}^s \lambda^2_i}\l\|u\r\|_2\r\}.
	\end{equation}
	This is a set of vectors with $s$ largest coordinates that ``dominate'' the remaining ones, in a sense made precise above. We will assume that the covariance matrix $\Sigma$ satisfies the following version of the \emph{restricted eigenvalue condition}: for any $u \in \mathcal{C}(s,4)$, 
\[
\|u\|^2_\Sigma \geq \kappa(s) \|u\|^2_2.
\]
In particular, if $\Sigma$ is non-degenerate, then it is always true that
\[
\kappa(s) \geq \lambda_{\min} (\Sigma)>0.
\]
	\item We will also be interested in a similar cone in the augmented space $\mb{R}^{p} \times \mb{R}^{n}$ that is defined via
\mln{\label{def:dimension reduction cone}
\mathcal{C}(s,c_0,o,\delta,\Sigma) = \l\{(u,v)\in\mb{R}^p \times \mb{R}^n: \l\|u\r\|_\lambda  + \l\|v\r\|_\mu \r.
\\ 
\leq \l. c_0 \l(\sqrt{\frac{\sum_{i=1}^s \lambda^2_i}{\kappa(s)} +\frac{\log(1/\delta)}{n}}\l\|u\r\|_\Sigma + \sqrt{\sum_{i=1}^o \mu^2_i}\l\|v\r\|_2 \r)\r\}.
}
	\item Finally, define the sequence 
	\[
	\lambda_i = C\sqrt{\frac{\log(ep/i)}{n}}, \ i=1,\ldots,p,
	\]
	where $C$ is a sufficiently large absolute constant. We will use the ordered $\|\cdot\|_{\lambda}$ norm corresponding to this sequence.
\end{itemize}

\subsection{Upper error bounds.}
\label{subsection:topic2_main}


Depending on the value of the parameter $\tau$ controlling the tails of the additive noise $\xi$, we will need to set the sequence $(\mu_n)_n$ differently (recall that $\tau$, or its lower bound, is assumed to be known). Specifically, let 
$$
\mu_i = \frac{C }{\sqrt{n}} \l(\frac{n}{i}\r)^{1/\tau}, \ i=1,\ldots,n
$$
and 
$$
\mu_i = C \sqrt{\frac{\log(en/i)}{n}}, \ i=1,\ldots, n
$$
for $\tau = \infty$.  Solution of the minimization problem \eqref{estimator:sqrt lasso} corresponding to this choice of penalization will be denoted via $\hat{\beta}_{\mathrm{sorted}}$.  
Similarly, given $0<\delta<1$, we denote by $\hat{\beta}_{\text{fixed}}$ the solution of \eqref{estimator:sqrt lasso} corresponding to 
$$
\mu_i = \frac{C }{\sqrt{n}} \l(\frac{n}{ \log(1/\delta)}\r)^{1/\tau},\ i=1,\ldots,n.
$$
Observe that both estimators $\hat{\beta}_{\mathrm{sorted}}$ and $\hat{\beta}_{\text{fixed}}$ are fully adaptive, the only requirement being the prior knowledge of $\tau$. In addition, notice that $\hat{\beta}_{\text{fixed}}$ requires the desired confidence level $\delta$ as an input while $\hat{\beta}_{\mathrm{sorted}}$ does not. 

\begin{theorem} 
\label{thm:sqrt lasso prob bound}
Assume that $\tau \geq  2$ and that assumptions \ref{ass:1}, \ref{ass:2} and \ref{ass:3} hold. There exist absolute positive constants $c,C'$ with the following properties: let $0<\delta<1$ be fixed, and assume that $s\log(p/s)/\kappa(s) + \log(1/\delta) + o \leq cn$. 
Then with probability at least $1-\delta$,
\mln{
\label{eq:result-1}
\|\wh{\beta}_{\text{fixed}} - \beta^*\|^2_\Sigma 
\leq C'\sigma^2 \l(\frac{s\log(ep/s) }{\kappa(s)n} \r.
\\
\l. + \l(\frac{o}{n}\r)^{2-2/\tau}\log(n/o) \l(1 + \l(\frac{o}{\log(1/\delta)}\r)^{2/\tau}\r) +\frac{\log(1/\delta)}{n} \r).
}
\end{theorem}
It follows from Theorem \ref{thm:lower} stated below that the bound \eqref{eq:result-1} is minimax optimal with respect to the contamination proportion $\frac{o}{n}$, up to the logarithmic factors, as long as $\log(1/\delta) \geq o$. Note that similar types of  conditions have appeared in the context of robust regression for methods based on the median of means estimator \cite{lecue2020robust}. In general, the condition $\log(1/\delta) \geq o$ is also required for robust mean estimation for instance using the trimmed mean \cite{lugosi2019robust} or self-normalized sums \cite{minsker2021robust}.
Finally, observe that \eqref{eq:result-1} is meaningful, although sub-optimal, even when $o\gg \log(1/\delta)$. 

\begin{theorem} 
\label{thm:sqrt lasso prob bound 2}
Assume that $\tau > 2$ and that assumptions \ref{ass:1},\ref{ass:2}  and \ref{ass:3} hold. There exist absolute positive constants $c,C'$ with the following properties: for any $\delta$ such that $s\log(p/s)/\kappa(s) + \log(1/\delta)+ o \leq cn$, the inequality 
\mln{
\label{eq:result-2}
\|\wh{\beta}_{\mathrm{sorted}} - \beta^*\|^2_\Sigma  \leq C'\sigma^2 \l(\frac{s\log(ep/s) }{\kappa(s)n} \r.
\\
\l.+ \l(\frac{o}{n}\r)^{2-4/\tau} + \l(\frac{\log(1/\delta)}{n}\r)^{2-4/\tau} + \frac{\log(1/\delta)}{n} \r)
}
holds with probability at least $1-\delta$. In particular, when the noise $\xi$ has sub-Gaussian distribution,
\[
\|\wh{\beta}_{\mathrm{sorted}} - \beta^*\|^2_\Sigma  \leq C'\sigma^2 \l(\frac{s\log(ep/s) }{\kappa(s)n} + \l(\frac{o\log(n/o)}{n} \r)^2 + \frac{\log(1/\delta)}{n} \r).
\]
\end{theorem}

For $\tau = \infty$, the bound implied by the inequality \eqref{eq:result-2} is minimax optimal up to the logarithmic factors. However, it is sub-optimal for $\tau<\infty$. Interestingly, the bound holds uniformly over the range confidence levels $e^{-cn}<\delta <1$ for the fixed choice of the regularization sequences $\{\lambda_i\}$ and $\{\mu_i\}$. 
In the special case where $o=0$ (no adversarial corruption) and $\tau \geq 4$, inequality \eqref{eq:result-2} yields a sub-Gaussian deviation bound with optimal dependence on the sparsity level $s$, despite the fact that the noise can be heavy-tailed.

\subsection{Lower error bounds.}
\label{sec:lower bound}

It is well known \cite[e.g. see][]{bellec2018slope} that in the absence of adversarial contamination and with $\Sigma = I_p$,
\[
\underset{\hat \beta}{\inf}\underset{|\beta|_0 \leq s}{\sup} \Prob\l(\|X(\hat{\beta} - \beta)\|_2^2/n \geq C' \sigma^2 \frac{s\log(ep/s)}{n}\r) \geq c
\]
for some positive constants $c,C'$.
In the Huber's contamination framework coupled with the assumption that the additive noise $\xi$ is Gaussian, results in \cite{chen2016general} yield that no estimator can achieve the error smaller than $C' \sigma^2 \l(\frac{s\log(ep/s)}{n} + \l(\frac{o}{n}\r)^2\r)$; of course, this lower bound is also valid for the adversarial contamination model. However, we could not find readily available lower bounds for the noise distributions beyond Gaussian. The following result gives an answer in this case. 
\begin{theorem}
\label{thm:lower}
Assume that at least one of the columns of $X$ belongs to $\{-1,1\}^n$. Then
\[
\underset{\hat \beta}{\inf}\,\underset{|\beta|_0 \leq s}{\sup}\,\underset{|\theta|_0 \leq o}{\sup}\,\underset{\sigma>0}{\sup}\, \underset{P_\xi \in \m P_{\tau,1}}{\sup} \,\Prob_{(\beta,\theta,\sigma,P_\xi)}\l(\|X(\hat{\beta} - \beta)\|^2/n \geq C \sigma^2 \l(\frac{o}{n}\r)^{2-2/\tau}\r) \geq c,
\]
for some $C,c>0$ where the infimum is taken over all measurable estimators. For sub-Gaussian noise the inequality takes the form
\[
\underset{\hat \beta}{\inf}\,\underset{|\beta|_0 \leq s}{\sup}\, \underset{|\theta|_0 \leq o}{\sup}\,\underset{\sigma>0}{\sup}\, \underset{P_\xi \in \m P_{\infty,1}}{\sup} \,\Prob_{(\beta,\theta,\sigma,P_\xi)}\l(\|X(\hat{\beta} - \beta)\|^2/n \geq C \sigma^2 \l(\frac{o}{n}\r)^{2}\log(n/o)\r) \geq c.
\]
\end{theorem}
The assumption on the design is very mild: indeed, it suffices that $\bf 1_n$ is a column of $X$, which is equivalent to including the intercept term in the regression. Another special case is the Rademacher design, implying that the lower bound holds for the class of sub-Gaussian design matrices.

\subsection{Main ideas of the proofs.}
\label{subsection:topic3_main}

In this section, we give a brief summary of the key ideas used in the proofs of Theorems \ref{thm:sqrt lasso prob bound} and \ref{thm:sqrt lasso prob bound 2}. We start by discussing several useful properties of sub-Gaussian design vectors. 

\begin{itemize}
	\item We show that the design matrix $X$ acts as a near-isometry on approximately sparse vectors: indeed, conditions of these type are crucial to guarantee success of sparse recovery. A detailed overview of similar assumptions that appear in the literature can be found in \cite{van2009conditions}. The specific form of the inequality that we prove is the following: with high probability, for all $u\in \mb R^p$ simultaneously
    \ben{
    \label{eq:property1}
    \frac{\|Xu\|^2_2}{n}\geq \frac{1}{2} \|u\|_\Sigma^2 - \|u\|^2_\lambda/4.
    }
	\item The columns of the design matrix $X$ need to be ``nearly uncorrelated'' with the columns of the identity matrix $I_n$. Indeed, consider a particular case where $n=p$ and $X=\sqrt{n} I_n$. In such a case, the model \eqref{model} becomes
	\[
	y_i = \sqrt{n} \l( \beta^* + \theta^*\r) + \sigma \xi_i,
	\]
	whence the only identifiable vector is $\beta^* + \theta^*$, making it impossible to consistently estimate $\beta^*$ itself. Assumptions of this type are commonly referred to as the \textit{incoherence conditions}. The incoherence employed in our proof takes the following form:  with probability at least $1-\delta$, for all $u,v\in \mb R^p$ simultaneously and some absolute constant $C'$,
   \ben{
       \label{eq:property2}
   \frac{1}{\sqrt{n}} |v^\top X u| \leq \|u\|_\lambda \|v\|_2/10 + \|v\|_\lambda \|u\|_\Sigma/10 + C' \sqrt{\frac{\log(1/\delta) +1}{n}}\|u\|_\Sigma\|v\|_2.
    }
Similarly,  for any fixed $v\in \mb R^n$, the following inequality holds with probability at least $1-\delta$ uniformly over all  $u\in \mb R^p$:
   \ben{
       \label{eq:property3}
   \frac{1}{\sqrt{n}} |v^\top X u| \leq \|u\|_\lambda \|v\|_2/10 + C' \sqrt{\frac{\log(1/\delta) +1}{n}}\|u\|_\Sigma\|v\|_2.
    }
In \cite{dalalyan2019outlier}, authors establish very similar conditions  for Gaussian design matrices.
\end{itemize}

We summarize the important properties of sub-Gaussian design matrices in the following result.			
\begin{theorem}
\label{thm:sub-gaussian}
Assume that $X = Y\Sigma^{1/2}$, where $Y$ has independent $1$-sub-Gaussian rows. Then, with probability at least $1-e^{-cn}$, $X$ satisfies \eqref{eq:property1}, and with probability at least $1-\delta$, $X$ satisfies \eqref{eq:property2} and \eqref{eq:property3}.
\end{theorem}
We note that properties \eqref{eq:property1}, \eqref{eq:property2} and \eqref{eq:property3} are the only conditions we require from the design. Next, we explain the way we deal with heavy-tailed noise. Fix an integer $o'\leq n$: we can then treat the largest, in absolute value, $o'$ coordinates of the noise vector $\xi$ as ``outliers'' that we merge with the vector $\theta^\ast$, while the remaining coordinates of $\xi$ are sufficiently well-behaved and ``light-tailed.'' Therefore, we can replace $\xi_{(i)}$ by $\xi_{(i)} \mathbf{1}\{i \geq o' \}$ and $\sqrt{n}\theta_{(i)}$ by $\sqrt{n}\theta_{(i)} + \xi_{(i)} \mathbf{1}\{i \leq o' \}$, where $\xi_{(i)}$ denotes the i-th largest, in absolute value, element of the vector $\xi$. Note that this new noise vector is no longer centered, and that $o + o'$ becomes a new upper bound of the number of outliers.   
We then define the ``good'' event $\m E$ via
\begin{align*}
\mathcal{E} = \l\{  n/10 \leq \sum_{j\geq o'} |\xi|^2_{(j)} \leq 2 n  \text{ and }  \forall j\geq o', |\xi|_{(j)} \leq \sqrt{n} \mu_{j}/20 \r\},
\end{align*}
and show that $\mathcal{E}$ holds with high probability (see Lemmas \ref{lem:variance} and \ref{lem:quantile}). 
The following inequality is our main result which in turn implies the bounds of Theorems \ref{thm:sqrt lasso prob bound} and \ref{thm:sqrt lasso prob bound 2} under various assumptions on $\xi$.
\begin{theorem}
\label{thm: sqrt lasso bound}
Fix any $o' \geq o$. There exist absolute positive constants $c,C'$ with the following properties: assume that $\sum_{i=1}^{s}\lambda^2_i /\kappa(s) + \log(1/\delta)/n+ \sum_{i=1}^{o'}\mu^2_i \leq c$ and that event $\mathcal{E}$ occurs. If moreover properties \eqref{eq:property1}, \eqref{eq:property2} and \eqref{eq:property3} of the design matrix hold, the following bound is valid  whenever $\lambda_i \leq \mu_i$:
\[
\|\wh{\beta} - \beta^*\|^2_\Sigma \leq C'\sigma^2\l( \frac{\sum_{i=1}^{s}\lambda^2_i}{\kappa(s)} + \underset{j\geq o'}{\max}(\lambda^2_j/\mu^2_j) \l(\sum_{j=1}^{o'} \mu_j^2\r)^2  + \frac{\log(1/\delta)}{n} \r).
\]
\end{theorem}

\section{Discussion and comparison to existing results.}
\label{section:discussion}

Below, we give a brief overview of existing literature and results that are most closely related to the problem considered in this work. For an extended overview of the classical and modern approaches to robust regression, we refer the reader to the excellent discussions in \cite[section 2]{she2011outlier} and \cite[section 4]{dalalyan2019outlier}. 

The idea of taking advantage of sparsity of the sequence of outliers and applying Lasso or Dantzig selector-type algorithms has been previously suggested in \cite{candes2008highly,laska2009exact,she2011outlier,nguyen2012robust,donoho2016high}, among other works. 
In particular, in \cite{donoho2016high} authors note that the solution $\widetilde\beta$ of the convex problem
\[
\l(\widetilde\beta,\widetilde\theta\r) = \argmin_{\beta\in \mb R^p,\theta\in \mb R^n} \l[\frac{1}{2n}\sum_{j=1}^n (y_j - X_j^T \beta - \theta)^2 + \lambda_1 \|\beta\|_1+ \lambda_2\|\theta\|_1\r]
\]
can be equivalently written, after carrying out minimization over $\theta$ explicitly, as 
\[
\widetilde \beta =  \argmin_{\beta\in \mb R^p} \l[ \lambda_2^2 \sum_{j=1}^n H\l( \frac{y_j - X_j^T \beta}{\lambda_2\sqrt{n}}\r) + \lambda_1 \|\beta\|_1 \r],
\] 
where 
\[
H(x)=\begin{cases}
x^2/2, & |x|\leq 1,
\\
|x|-1/2, & |x|>1
\end{cases}
\]
is the Huber's loss function; similar connection has been used in several earlier works, including \cite{sardy2001robust,gannaz2007robust}. More recently, in \cite{dalalyan2019outlier,thompson2020outlier} authors improved the  bounds proven in \cite{nguyen2012robust} and showed that for the Gaussian design and Gaussian additive noise, 
$\|\widetilde\beta - \beta^\ast\|^2_\Sigma = O_P\l( \frac{s\log(p)}{n} + \l(\frac{o\log(n)}{n}\r)^2\r)$ which is nearly minimax optimal in the ratio $\frac{o}{n}$ (note that the additional $\log(n)$ factor makes the bound suboptimal \cite{chen2016general}). At the same time, estimators that achieve minimax optimality, such as the methods based on regression depth \cite{gao2020robust}, are not computationally feasible, 
Our work has been partially motivated by the question raised by the authors of \cite{dalalyan2019outlier}, namely, whether the penalized ERM-type methods can also handle the case of heavy-tailed additive noise variables $\{\xi_i\}_{i=1}^n$ and yield optimal or near-optimal rates. Results of the present paper give a generally affirmative answer and make an extra step by proving that it is possible to be computationally efficient and minimax optimal with respect to the sparsity level and contamination level, while being completely adaptive and achieve strong concentration of the resulting estimators simultaneously. Related results in the literature, such as the work \cite{finocchio2021robust}, establish strong theoretical guarantees for the estimators that are not efficiently computable. 

Very recently, a model with adversarially contaminated design and response was considered in \cite{sasai2022outlier}, however, the resulting bounds are only valid for very sparse signals such that $s\lesssim \sqrt{n}$. A similar setup was also considered in \cite{diakonikolas2019efficient} and \cite{pensia2020robust} without the sparsity assumptions. For example, in \cite{pensia2020robust} the authors used a black-box ``filtering'' algorithm to eliminate outliers from the design matrix provided that the covariance matrix $\Sigma$ is known. Our goal was to show that similar results hold for a simple procedure and without additional knowledge about the parameters of the problem. Finally, let us remark that in the low-dimensional case $p<n$, there exist estimators capable of approximating $\beta^\ast$ regardless of the number of outliers $o$ as long the following conditions hold: (i) $o<cn$, (ii) the contamination is oblivious and (iii) the design matrix $X$ is sufficiently nice (e.g., has normally distributed rows); this fact was proven in \cite{bhatia2017consistent}.

\section*{Acknowledgements.}
Stanislav Minsker and Lang Wang acknowledge support by the National Science Foundation grants CIF-1908905 and DMS CAREER-2045068. The work of Mohamed Ndaoud was supported by a Chair of
Excellence in Data Science granted by the CY Initiative.

\bibliographystyle{plain}

\appendix
\input{appendix_proof}


\end{document}

%% file: appendix_proof.tex


\section{Technical results.}
\label{app:lemma}

Recall that $\E(|\xi_i|^\tau) \leq a^\tau$ for some $\tau \geq 2$ and $a\geq 1$, implying that $\Prob(|\xi| \geq t) \leq (a/t)^\tau \text{ for all } t$. Without loss of generality, we will assume that $a=1$, otherwise we can simply replace $\xi_i$ by $\xi_i/a$ and $\sigma$ by $\sigma\cdot a$. 

\begin{lemma}
\label{lem:quantile}
For any $1\leq i\leq n$, set $\mu_i = \frac{C}{\sqrt{n}} \l(\frac{n}{i}\r)^{1/\tau}$ for $\tau \geq 2$ and $C\geq 80$.  Then for any $k \geq 1$, the following inequality holds:
\[
\Prob\l(\underset{i \geq k}{\max}\frac{|\xi|_{(i)}}{\sqrt{n}\mu_i}  \geq 1/20 \r) \leq 2e^{-k}.
\]
Moreover, for all $i$ such that $\log(n) \leq i\leq n$, 
\[
\E\l(\frac{|\xi|_{(i)}}{\sqrt{n}\mu_i}\r) \leq 1. 
\]
The result holds for sub-Gaussian noise as well with the choice $\mu_i = \lambda_i$.
\end{lemma}
\begin{proof}
For any fixed $i\geq k$,
\[
\Prob\l(\frac{|\xi|_{(i)}}{\sqrt{n}\mu_i}  \geq 1/20 \r) = \Prob\l(\exists |I| = i, \forall j \in I \quad   \frac{|\xi_j|}{\sqrt{n}\mu_i}  \geq 1/20 \r).
\]
Therefore, applying the inequality ${n \choose i} \leq e^{i\log(en/i)}$ and the assumption $C\geq 80$, we deduce that 
\ben{
\label{eq:prob_exp}
\Prob\l(\frac{|\xi|_{(i)}}{\sqrt{n}\mu_i}  \geq 1/20 \r) \leq  e^{i\log(en/i)}\Prob\l(  \frac{|\xi_1|}{\sqrt{n}\mu_i}  \geq 1/20 \r)^{i}\leq e^{-i}.
}
We conclude using the union bound over $i\geq k$ and the fact that $\sum_{i=k}^n e^{-i} \leq 2e^{-k}$. 

\noindent To get the result in expectation, let us denote $ \gamma := \frac{|\xi|_{(i)}}{\sqrt{n}\mu_i} $. Observe that
\[ \E(\gamma) \leq \E(\gamma \mathbf{1}\{\gamma \leq 1/20\}) + \E(\gamma \mathbf{1}\{\gamma \geq 1/20\}). \]
Using Cauchy-Schwarz inequality, we get that
\[
\E\l(\frac{|\xi|_{(i)}}{\sqrt{n}\mu_i}\r) \leq 1/20 + \sqrt{\E\l(\frac{|\xi|^2_{(i)}}{n\mu^2_i}\r)\Prob\l(\frac{|\xi|_{(i)}}{\sqrt{n}\mu_i}\geq 1/20\r)}. 
\]
Since $n\mu^2_i \geq 8$ and $ \E(|\xi|^2_{(i)}) \leq  \E\l(\sum_{i=1}^{n}\xi^2_{i} \r) \leq n  $ we conclude using \eqref{eq:prob_exp} that
\[
\E\l(\frac{|\xi|_{(i)}}{\sqrt{n}\mu_i}\r) \leq 1/20 + \frac{\sqrt{n}}{2}e^{- i /2 } \leq 1,
\]
as long as $i \geq \log(n)$.
\end{proof}

\begin{lemma}
\label{lem:variance}
Assume that $\E(\xi^2_i \mathbf{1}\{|\xi_i|\leq 1/2\} ) \geq 1/4 $ and that $o \leq n/1000$. Then
\[
\Prob\l( n/10 \leq \sum_{i=o}^{n} |\xi|^2_{(i)}  \leq 2n \r) \geq 1 - 3e^{-co},
\]
for an absolute constant $c>0$.
\end{lemma}
\begin{proof}
For the upper bound, we only need to control the random variables bounded by $C\sqrt{n/o}$, in view of Lemma \ref{lem:quantile} applied with $k=o$. Set
\[
R = C\left( \frac{n}{o }\right)^{1/2}. 
\]
Observe that, as long as $|\xi|_{(o)} \leq R $, we have
\[
\sum_{i=o}^{n} |\xi|^2_{(i)} \leq \sum_{i=1}^{n} \xi^2_{i}\mathbf{1}\{|\xi_i| \leq R\}.
\]
Since $\xi^2_{i}\mathbf{1}\{|\xi_i| \leq R\} \leq R^2$ and $\E(\xi^2_{i}\mathbf{1}\{|\xi_i| \leq R\}) \leq 1$, Hoeffding's inequality yields that 
\[
\Prob\l( \l\{ \sum_{i=o}^{n} |\xi|^2_{(i)}   \geq 2 n \r\} \cap \l\{ |\xi|_{(o)} \leq R   \r\} \r) \leq \exp(-  n/R^2) .
\]
Noticing that $n/R^2 = o/C^2$, the upper bound follows for $c = 1/C^2$ from the inequality \eqref{eq:prob_exp}, since
\[
\Prob\l( \sum_{i=o}^{n} |\xi|^2_{(i)}   \geq 2 n  \r) \leq \Prob\l( \l\{ \sum_{i=o}^{n} |\xi|^2_{(i)}   \geq 2 n \r\} \cap \l\{ |\xi|_{(o)} \leq R   \r\} \r) + \Prob\l(  |\xi|_{(o)} \geq R  \r) \leq 2\exp(-co) .
\]
For the lower bound, observe that 
\[
\Prob\l( n/10 \geq \sum_{i=o}^{n} |\xi|^2_{(i)}   \r) = \Prob\l( n/10 \geq \underset{|I| = n-o+1}{\min}\sum_{i \in I} |\xi_i|^2  \r) \leq e^{o\log(\frac{en}{o-1})} \Prob\l( n/10 \geq \sum_{i=1}^{n-o+1} |\xi_i|^2  \r),
\]
where we use the union bound together with the relations ${n \choose n-o+1}  = {n \choose o-1} \leq e^{o\log(\frac{en}{o-1})}$.  
Since $\E(|\xi_i|^2 \mathbf{1}\{|\xi_i| \leq 1/2 \}) \geq 1/4$,
\mln{
\Prob\l( n/10 \geq \sum_{i=o}^{n} |\xi|^2_{(i)}   \r)  \\
\label{eq:hoeffding1}
\leq e^{o\log(en/(o-1))} \Prob\l( - n/10 \geq \sum_{i=1}^{n-o+1} |\xi_i|^2 \mathbf{1}\{|\xi_i| \leq 1/2 \} 
 - \E(|\xi_i|^2 \mathbf{1}\{|\xi_i| \leq 1/2 \})  \r).
}
We conclude, using Hoeffding's inequality to estimate the probability in \eqref{eq:hoeffding1}, that
\[
\Prob\l( n/10 \geq \sum_{i=o}^{n} |\xi|^2_{(i)}   \r)  \leq e^{o\log(en/(o-1)) - n/100} \leq \exp(-cn),
\]
for $c$ small enough. 
\end{proof}

\begin{lemma}
\label{lem:control_max}
Assume that $\xi$ is a centered Gaussian vector such that $\E(\xi_i^2)\leq 1$ for all $1\leq i\leq n$. Then
\[
\E\l( \underset{i=1,\dots,n}{\max} \frac{|\xi|_{(i)}}{\sqrt{\log(en/i)}}\r) \leq 20.
\]
\end{lemma}
\begin{proof}
Set $\lambda_i^2 = 4\log(en/i)$. Let $\hat{i}$ be an index such that
$ \underset{i=1,\dots,n}{\max} \frac{|\xi|_{(i)}}{\lambda_i} = \frac{|\xi|_{(\hat i)}}{\lambda_{\hat i}}$. Then
\begin{align*}
    \E\l( \underset{i=1,\dots,n}{\max} \frac{|\xi|_{(i)}}{\lambda_i}\r)  &\leq 1 + \E\l( \frac{|\xi|_{(\hat i)}}{\lambda_{\hat i}} \mathbf{1}\l( \frac{|\xi|_{(\hat i)}}{\lambda_{\hat i}} \geq 1\r)\r)\\
    & \leq 1 + \int_1^{\infty} \Prob\l( |\xi|_{(\hat i)} \geq t \lambda_{\hat i} \r) \mathrm{d}t\\
    & \leq 1 + \int_1^{\infty} \Prob\l( \hat i  \exp\l( \xi^2_{(\hat i)}/4 \r) \geq \hat i \exp\l(t^2 \lambda^2_{\hat i}/4\r) \r) \mathrm{d}t.
\end{align*}
On the one hand, we have that 
\[
 \hat i  \exp\l( \xi^2_{(\hat i)}/4 \r) \leq \sum_{j=1}^i \exp\l( \xi^2_{( j)}/4 \r) \leq \sum_{j=1}^n \exp\l( \xi^2_{ j}/4 \r).
\]
On the other hand, for $t^2\geq 1$
\[
\hat i \exp\l(t^2 \lambda^2_{\hat i}/4\r) = \hat i\l(en/ \hat i \r)^{t^2} \geq  n e^{t^2}.
\]
Therefore,
\[
\E\l( \underset{i=1,\dots,n}{\max} \frac{|\xi|_{(i)}}{\lambda_i}\r) \leq 1 +  \int_1^{\infty} \Prob\l( \sum_{j=1}^n \exp\l( \xi^2_{ j}/4 \r) \geq n e^{t^2} \r) \mathrm{d}t.
\]
Since $\E\l(\exp\l( \xi^2_{ j}/4 \r)\r) \leq 5$, we conclude using Markov's inequality that
\[
\E\l( \underset{i=1,\dots,n}{\max} \frac{|\xi|_{(i)}}{\lambda_i}\r) \leq 1 +  5\int_1^{\infty}  e^{-t^2} \mathrm{d}t\leq 10.
\]
Re-scaling $\lambda_i$ by $2$ yields the result.
\end{proof}

\section{Proofs of the main results.}
\label{app:main}

The proof of the lower bound in inspired by results in \cite{comminges2021adaptive} where the goal was to estimate the nuisance parameter $\theta$ rather than the signal $\beta$ itself.
\subsection{Proof of Theorem \ref{thm:lower}.}
\label{app:lower_bound}
Assume that the first column $v$ of $X$ is such that $v\in \{\pm1\}^n$. Let us choose $\beta$ proportional to the canonical basis vector $e_1$ (recall that $\beta$ is sparse) such that $X\beta = \frac{\|X\beta\|_2}{\sqrt{n}}v$. Moreover, let $\xi$ be a vector of i.i.d. Rademacher random variables. Clearly, $P_\xi \in \m P_{\tau,1}$ for all values of $\tau$. The vector $\theta$ will be chosen to be random with i.i.d entries such that
$
\theta_i = \sigma \l(\frac{o}{n}\r)^{-1/\tau}\alpha_i v_{i},
$ where $\alpha_i$ are i.i.d. Bernoulli random variables with parameter $o/n$. Therefore,
\[
\E(\theta_i) = \sigma \l(\frac{o}{n}\r)^{1-1/\tau}v_{i} 
\text{  and  } \var(\theta_i) = \sigma^2 \l(\frac{o}{n}\r)^{1-2/ \tau}\l(1-\l(\frac{o}{n}\r)^{1-2/ \tau}\r) \leq \sigma^2\l(\frac{o}{n}\r)^{1-2/ \tau}.
\]
Notice that $\theta$ is not exactly of sparsity less than $o$ but we will deal with this technicality exactly as in \cite{comminges2021adaptive}. Finally, set 
\[
\frac{\|X\beta\|_2}{\sqrt{n}} = \sigma \l(\frac{o}{n}\r)^{1-1/\tau}.
\]
Notice that 
\[
Y_i = (X\beta - \theta +\sigma \xi)_i = -(\theta_i - \E(\theta)_i)v_i + \sigma \xi_i.
\]
Hence, the distributions of $Y$ defined by the model corresponding to $(\beta,-\theta,\sigma,P_\xi)$ and $(0,0,\tilde{\sigma},\tilde{P}_\xi)$ are identical.
Here, $\tilde{\sigma}^2 = \sigma^2(1+ \l(\frac{o}{n}\r)^{1-2/\tau}(1-\l(\frac{o}{n}\r)^{1-2/\tau})) \sim \sigma^2$ and $\tilde{P}_\xi$ is the distribution of $\zeta = \sigma \l(\frac{o}{n}\r)^{-1/\tau}((\alpha_i - \l(\frac{o}{n}\r))v_i + \l(\frac{o}{n}\r)^{\tau}\xi_i)/\tilde{\sigma}$. Notice that $|\zeta| \geq 2$ only if $\alpha_i = 1$, hence for all $2 \leq t \leq (o/n)^{ - 1/ \tau}$ we have that
\[
\Prob( |\zeta| \geq t) = o/n \leq \l(\frac{1}{t}\r)^{\tau},
\]
and for $t > \l(\frac{o}{n}\r)^{ - 1/ \tau}$,
\[
\Prob( |\zeta| \geq t) = 0 \leq \l(\frac{1}{t}\r)^{\tau}.
\]
Therefore, $\tilde{P}_\xi \in \m P_{\tau,1}$. Let us denote 
\[
\mathcal{R}^* = \underset{\hat \beta}{\inf}\underset{|\beta|_0 \leq s}{\sup}\underset{\mbox{ }|\theta|_0 \leq o}{\sup}\underset{\mbox{ }\sigma>0}{\sup} \underset{\mbox{ } P_\xi \in \m P_{\tau,1}}{\sup} \Prob_{(\beta,\theta,\sigma,P_\xi)}\l(\|X(\hat{\beta} - \beta)\|^2/n \geq  \sigma^2/16 \l(\frac{o}{n}\r)^{2-2/\tau}\r).
\]
It is easy to notice that
\ml{
\mathcal{R}^* \geq \underset{\hat T}{\inf} \l(   \Prob_{(0,0,\tilde{\sigma},\tilde{P}_\xi)}\l(|\hat{T}| \geq  \tilde{\sigma}/4 \l(\frac{o}{n}\r)^{1-1/\tau}\r) \r. 
\\
\l.\bigvee \Prob_{(\beta,-\theta,\sigma,P_\xi)}\l(\l|\hat{T} - \frac{\|X\beta\|_2}{\sqrt{n}}\r| \geq  \sigma/4 \l(\frac{o}{n}\r)^{1-1/\tau}\r) \r),
}
where $\hat T$ is an estimator of $\frac{\|X\beta\|_2}{\sqrt n}$. 
Since $\tilde{\sigma} \geq \sigma$ and the distributions $\Prob_{(0,0,\tilde{\sigma},\tilde{P}_\xi)}, \ \Prob_{(\beta,-\theta,\sigma,P_\xi)}$ are equal, we deduce that
\[
\mathcal{R}^* \geq \underset{\hat T}{\inf} \l(   \Prob\l(|\hat{T}| \geq  \tilde{\sigma}/4 \l(\frac{o}{n}\r)^{1-1/\tau}\r)\vee \Prob\l(|\hat{T}| \leq  \tilde{\sigma}/4 \l(\frac{o}{n}\r)^{1-1/\tau}\r) \r),
\]
as long as $ \frac{\|X\beta\|}{\sqrt{n}} \geq \frac{\tilde{\sigma}\l(\frac{o}{n}\r)^{1-1/\tau}}{2}$. The last condition is satisfied since $\tilde{\sigma} \leq 2\sigma$. We conclude that
\[
\mathcal{R}^* \geq 1/2.
\]
In the case of sub-Gaussian noise, we choose $\theta_i = \sigma\sqrt{\log\l(\frac{o}{n}\r)} \alpha_i v_i $ and follow the same argument.

\subsection{Proof of Theorem \ref{thm:sub-gaussian}.}
\label{app:sub-Gaussian}
We start with the property given by the inequality \eqref{eq:property1}. We will show first that
for all vectors $u$,
\[
\frac{\|Xu\|^2_2}{n}\geq \frac{1}{2} \|u\|_\Sigma^2 - \|u\|^2_\lambda/4,
\]
where $\|u\|_\Sigma^2 = u^\top \Sigma u$. Define $\tilde{X}$such that $X = \tilde{X}\Sigma^{1/2}$, let $A$ be the set  
 \[
 A = \l\{ u: \ \| u \|_\Sigma^2 \geq  \|u \|^2_\lambda /2 \r\},
 \]
 and $K = \{  \Sigma^{1/2}u, \ u \in A\}$. Moreover, we will denote the sphere of radius $1$ in $\mb R^p$ via $S^{p-1}$. 
 Since $\tilde{X}$ is isotropic, Corollary $1.5$ in \cite{liaw2017simple} implies that for any $h \geq 0$ and for all $v \in K\cap S^{p-1}$
 \[
 \frac{\|\tilde{X}v\|_2}{\sqrt{n}} \geq 1 - C'\l(\frac{\omega(K\cap S^{p-1}) + h}{\sqrt n}\r),
 \]
 with probability at least $1-e^{-ch^2}$. Here, $C'$ is a positive absolute constant and $\omega(T)$ corresponds to the Gaussian mean width of $T\subseteq \mb R^p$ defined via 
 \[
 \omega(T) = \mb E\sup_{u,v\in T} \langle \xi, u-v \rangle
 \] 
where $\xi$ has standard normal law. It is clear that 
 \[
 \omega(K\cap S^{p-1}) = \E\l(\underset{ v \in K\cap S^{p-1}}{\sup} \xi^\top v\r) = \E\l(\underset{ u \in A}{\sup} \; \xi_\Sigma^\top u / \|u\|_{\Sigma}\r),
 \]
where $\xi_\Sigma$ is a centered Gaussian random vector with covariance $\Sigma$. Therefore,
\ben{
 \label{eq:control_max}
       \omega(K\cap S^{p-1}) \leq \E\l( \underset{i}{\max} \frac{|\xi_\Sigma|_{(i)}}{\lambda_i} \r) \underset{ u \in A}{\sup} \|u\|_\lambda/\|u\|_{\Sigma}.
}
Since diagonal elements of $\Sigma$ do not exceed $1$ and $\lambda_i = C\sqrt{\frac{\log(ep/i)}{n}}, \ i=1,\ldots,p,$ we deduce from the bound of Lemma \ref{lem:control_max} that whenever $C$ in the definition of $\lambda$ is large enough,
 \[
 \omega(K\cap S^{p-1}) \leq \frac{1}{10C'}.
 \]
 Hence, for all $u \in A$
 \[
   \frac{\|Xu\|_2}{\sqrt{n}} \geq \|u\|_\Sigma/\sqrt{2},
 \]
 with probability at least $1-e^{-cn}$. This concludes the first part of the proof, since the inequality is always true for $u \not\in A$. We will now prove inequality \eqref{eq:property3}. Fix $v\in \mb R^p$. We want to show that for all $u\in \mb R^p$, 
    \[
   \frac{1}{\sqrt{n}} |v^\top X u| \leq \|u\|_\lambda \|v\|_2/10 + C \sqrt{\frac{1+\log(1/\delta)}{n}}\|u\|_\Sigma\|v\|_2.
    \]
   This is equivalent to establishing that for all $u\in \mb R^p$, 
   \[
   \frac{1}{\sqrt{n}} \xi^\top \Sigma^{1/2} u \leq \|u\|_\lambda /10 + C \sqrt{\frac{1+\log(1/\delta)}{n}}\|u\|_\Sigma,
    \]
    where $\xi$ is an isotropic sub-Gaussian vector. For $\alpha>0$, let $A_\alpha$ be the set
    \[
    A_\alpha = \{ u: \  \|u\|_\Sigma = 1, \quad \|u\|_\lambda \leq \alpha \},
    \]
    and let $K_\alpha$ be the set $K_\alpha = \{\Sigma^{1/2}u , \ u\in A_\alpha \}$. Notice that $K_\alpha \subset S^{p-1}$ and that 
    \[
    \underset{u \in A_\alpha }{\sup} \frac{1}{\sqrt{n}} \xi^\top \Sigma^{1/2} u = \underset{v \in K_\alpha }{\sup} \frac{1}{\sqrt{n}} \xi^\top v. 
    \]
    Hence, applying Theorem 4.1 in \cite{liaw2017simple} on $K_\alpha$, we get that
    \[
    \underset{u \in A_\alpha }{\sup} \frac{1}{\sqrt{n}} \xi^\top \Sigma^{1/2} u \leq C'\l( \E\l( \underset{u \in A_\alpha }{\sup} \frac{1}{\sqrt{n}} {\xi}^\top \Sigma^{1/2} u \r) + \sqrt{\frac{\log(1/\delta)}{n}} \r)
    \]
    for some $C'>0$ with probability $1-\delta$, where ${\xi}$ is a standard Gaussian random vector. Using the bound \eqref{eq:control_max} we deduce that the inequality
    \[
    \underset{u \in A_\alpha }{\sup} \frac{1}{\sqrt{n}} \xi^\top \Sigma^{1/2} u \leq \alpha/20 + C'\sqrt{\frac{\log(1/\delta)}{n}}
    \]
holds with probability at least $1-\delta$. We can now conclude, using the peeling argument as in \cite[Lemma 5]{dalalyan2019outlier} that
 \[
 \underset{\|u\|_\Sigma = 1  }{\sup} \frac{1}{\sqrt{n}} \xi^\top \Sigma^{1/2} u \leq \|u\|_\lambda /10 + C'\sqrt{\frac{1+\log(1/\delta)}{n}},
 \]
 again with probability at least $1-\delta$. The proof is complete by homogeneity of the norm. For the remaining part of the proof, we need to show that for all $u,v$
    \[
   \frac{1}{\sqrt{n}} |v^\top \tilde{X}\Sigma^{1/2} u| \leq \|u\|_\lambda \|v\|_2/10 + \|v\|_\lambda \|u\|_\Sigma/10 + C \sqrt{\frac{1+\log(1/\delta)}{n}}\|u\|_\Sigma\|v\|_2,
    \]
    with probability at least $1-\delta$. For $\alpha,\beta>0$, let $A_\alpha$ and $B_\beta$ be the sets
    \[
    A_\alpha = \{ u: \   \|u\|_\Sigma = 1, \ \|u\|_\lambda \leq \alpha \},
    \]
    and
    \[
    B_\beta = \{ v: \  \|v\|_2 = 1, \ \|v\|_\lambda \leq \beta \},
    \]
    and let $K_\alpha$ be the set $K_\alpha = \{\Sigma^{1/2}u , \ u\in A_\alpha \}$. Notice that $K_\alpha \subset S^{p-1}$ and that 
    \[
    \underset{(u,v) \in A_\alpha \times B_\beta }{\sup} \frac{1}{\sqrt{n}} v^\top \tilde{X}\Sigma^{1/2} u= \underset{(b,v) \in K_\alpha \times B_\beta }{\sup} \frac{1}{\sqrt{n}} v^\top \tilde{X}b. 
    \]
    Let us denote by $Z_{v,b}$ the sub-Gaussian process $v^\top \tilde{X}b$ where $v,b$ are both of norm $1$. We see that 
\ml{
\E\l( Z_{v,b} - Z_{v',b'}  \r)^2 = 2(1 - \langle v, v' \rangle   \langle b,b' \rangle) 
\leq 4 - 2(\langle v, v' \rangle  + \langle b,b' \rangle)
\\
\leq \|v - v'\|^2 + \|b - b'\|^2.
}
    Therefore,
    \[
    \E\l( Z_{v,b} - Z_{v',b'}  \r)^2 \leq \E\l( \xi^\top(v-v')   \r)^2 + \E\l( \tilde{\xi}^\top(b-b')   \r)^2,
    \]
    where $\xi $ and $\tilde{\xi}$ are both standard Gaussian vectors.
    Applying Theorem 4.1 in \cite{liaw2017simple} on $K_\alpha \times B_\beta$, we deduce that 
    \ml{
    \underset{(u,v) \in A_\alpha \times B_\beta }{\sup} \frac{1}{\sqrt{n}} v^\top \tilde{X}\Sigma^{1/2} u \leq C'\l( \E\l( \underset{u \in A_\alpha }{\sup} \frac{1}{\sqrt{n}} \tilde{\xi}^\top \Sigma^{1/2} u \r) + \E\l( \underset{v \in B_\alpha }{\sup} \frac{1}{\sqrt{n}} \tilde{\xi}^\top  v \r)\r.
    \\
    \l.+ \sqrt{\frac{\log(1/\delta)}{n}} \r)
    }
    for some $C'>0$ with probability at least $1-\delta$. 
   Using the inequality \eqref{eq:control_max}, we get that
    \[
    \underset{(u,v) \in A_\alpha \times B_\beta }{\sup} \frac{1}{\sqrt{n}} v^\top \tilde{X}\Sigma^{1/2} u \leq \alpha/20 + \beta/20 + C'\sqrt{\frac{\log(1/\delta)}{n}},
    \]
 with probability at least $1-\delta$. We can now conclude, using the double-peeling argument as in \cite[Lemma 6]{dalalyan2019outlier} that
 \[
 \underset{\|u\|_\Sigma = 1, \|v\|_2=1  }{\sup}  \frac{1}{\sqrt{n}}v^\top \tilde{X}\Sigma^{1/2} u\leq \|u\|_\lambda /10+ \|v\|_\lambda /10 + C'\sqrt{\frac{1+\log(1/\delta)}{n}},
 \] 
 with probability at least $1-\delta$. The desired result now follows by homogeneity of the norm.

\begin{proposition}
\label{prop:RE}
 Let $X$ be a matrix with sub-Gaussian rows. Then for all vectors $u,v$
    \[
   \|Xu/\sqrt{n} + v\|^2_2\geq \frac{1}{8}  \|u\|_\Sigma^2 + \frac{1}{8}\|v\|_2^2 - \|u\|^2_\lambda/2 - \|v\|^2_\mu/2,
    \]
    with probability at least $1-e^{-cn}$.
    \end{proposition}
\begin{proof}
We have
\[
\|Xu/\sqrt{n} + v\|^2_2\geq \|Xu/\sqrt{n} \|^2_2 + \| v\|^2_2 - \frac{2}{\sqrt{n}} |v^\top X u|.
\]
We conclude using the inequalities \eqref{eq:property1} and \eqref{eq:property2} with $\delta = e^{-cn}$ for $c$ small enough.
\end{proof}

\subsection{Proof of Theorem \ref{thm: sqrt lasso bound}.}
\label{proof of thm: sqrt lasso bound 2}

Throughout this section, we set $\Delta^\beta = \wh{\beta}-\beta^*$, $\Delta^\theta = \wh{\theta}-\theta^*$ and let $\Delta = [\Delta^\beta; \Delta^\theta] \in \mb{R}^{p+n}$ be the augmented error vector. 
We also introduce the following additional notation with the goal of simplifying the expressions; recall that 
\[
Q(\beta,\theta) := \frac{1}{2n}\l\|Y-X\beta-\sqrt{n}\theta\r\|_2^2,
\] 
and let 
\begin{itemize}
	\item $\wh{Q}:=Q(\wh{\beta},\wh{\theta})$ and $Q^*:= Q(\beta^*,\theta^*)$;
	\item $A^{(n)} := \frac{1}{\sqrt{n}} A$, whenever $A$ is a scalar, vector or a matrix;
	\item $\wh{\xi} := Y - X \wh{\beta} - \sqrt{n}\wh{\theta}$.
\end{itemize}
Moreover, note that $\wh{Q} =  \frac{1}{2n} \l\| \wh{\xi} \r\|_2^2 $ and $Q^* = \frac{1}{2n} \l\|\xi \r\|_2^2$. In the rest of the proof we assume that the event
\[
\mathcal{E} = \l\{  n/10 \leq \sum_{j\geq o'} \xi^2_{(j)}\leq 2 n  \text{ and }  \forall j\geq o', |\xi|_{(j)} \leq \sqrt{n} \mu_{(j)}/20 \r\}
\]
occurs and that Properties $1,2$ and $3$ of the design matrix hold as well.

Given $o'> o$, let us replace the dense noise vector $\xi$ by the new vector obtained from $\xi$ by replacing the largest $o'$ coordinates by $0$. We will treat the largest $o'$ coordinates removed from $\xi$ as a subset of adversarial outliers, and will  replace the corruption parameter $o$ by $2o'$. From now on, we can assume that the entries of the ``new'' noise vector $\xi$ are bounded by $|\xi|_{(o')}$.

Let us note that several steps of the proof below follow the argument in \cite{dalalyan2019outlier}.
First, recall that $S = \big\{j: \beta^*_j \neq 0\big\}$, $O = \big\{j: \theta^*_j \neq 0\big\}$ and $s= \card(S)$, $ \card(O) \leq 2 o' $. The following lemma ensures that the augmented error vector $\Delta$ belongs to 
the cone defined in \eqref{def:dimension reduction cone}, with $c_0=4$.
\begin{lemma}
\label{lemma:sparsity inequality}
The following inequality holds:
\[
\|\Delta^\beta\|_\lambda  + \|\Delta^\theta\|_\mu  \leq 4\sqrt{\frac{\sum_{i=1}^s \lambda_i^2}{\kappa(s)} + \frac{\log(1/\delta)}{n}} \|\Delta^\beta\|_\Sigma + 4\sqrt{\sum_{i=1}^{o'} \mu_i^2} \|\Delta^\theta\|_2.
\]
Equivalently, $\Delta \in \mathcal{C}\l(4,s,o',\delta,\Sigma \r)$.
\end{lemma}

\begin{proof}[Proof of Lemma \ref{lemma:sparsity inequality}]
By the definition of $\wh{\beta},\wh{\theta}$, we have that
\begin{equation}\label{lemma 4.1 proof eq1}
\wh Q^{\frac{1}{2}} - {Q^\ast}^{\frac{1}{2}} \leq \l(\l\|\beta^*\r\|_\lambda - \l\|\wh\beta\r\|_\lambda\r) + \l( \l\|\theta^*\r\|_\mu-\l\|\wh\theta\r\|_\mu \r).
\end{equation}
Using Lemma A.1 in \cite{bellec2018slope}, we get that
\[
\l(\l\|\beta^*\r\|_\lambda - \l\|\wh\beta\r\|_\lambda\r) \leq 2\sqrt{\sum_{i=1}^s \lambda_i^2} \|\Delta^\beta\|_2 - \|\Delta^\beta\|_\lambda.
\]
If $\Delta^\beta \in \mathcal{C}(s,4) $, then 
\[
\l(\l\|\beta^*\r\|_\lambda - \l\|\wh\beta\r\|_\lambda\r) \leq 2\sqrt{\frac{\sum_{i=1}^s \lambda_i^2}{\kappa(s)}} \|\Delta^\beta\|_\Sigma - \|\Delta^\beta\|_\lambda.
\]
Otherwise,
\[
\l(\l\|\beta^*\r\|_\lambda - \l\|\wh\beta\r\|_\lambda\r) \leq - \|\Delta^\beta\|_\lambda/2.
\]
In both cases we have that
\begin{equation}
\l(\l\|\beta^*\r\|_\lambda - \l\|\wh\beta\r\|_\lambda\r) \leq 2\sqrt{\frac{\sum_{i=1}^s \lambda_i^2}{\kappa(s)}} \|\Delta^\beta\|_\Sigma - \|\Delta^\beta\|_\lambda/2.
\end{equation}
Similarly, we observe that
\[
\l(\l\|\theta^*\r\|_\mu - \l\|\wh\theta\r\|_\mu\r) \leq 2\sqrt{\sum_{i=1}^{o'} \mu_i^2} \|\Delta^\theta\|_2 - \|\Delta^\theta\|_\mu.
\]
Combining the inequalities above with \eqref{lemma 4.1 proof eq1}, we deduce that
\begin{multline}
\label{Q_hat_Q_star_upper_bound}
\wh Q^{\frac{1}{2}} - {Q^\ast}^{\frac{1}{2}} \leq 2\sqrt{\frac{\sum_{i=1}^s \lambda_i^2}{\kappa(s)}} \|\Delta^\beta\|_\Sigma + 2\sqrt{\sum_{i=1}^{o'} \mu_i^2} \|\Delta^\theta\|_2 - (\|\Delta^\beta\|_\lambda + \|\Delta^\theta\|_\mu)/2.
\end{multline}
Recall that the property \eqref{eq:property3} of sub-Gaussian designs together with the inequality $\|\xi\| \leq 2\sqrt{n}$ yields the bound
\begin{equation}
\label{eq:test arxiv 1}
      \frac{1}{n} |\xi^\top X \Delta^\beta| \leq \|\Delta^\beta\|_\lambda/10  +  2\sqrt{\frac{\log(1/\delta)}{n}}\|\Delta^\beta\|_\Sigma. \end{equation}
On the other hand, convexity of $Q(\beta,\theta)^{\frac{1}{2}}$ implies that
\[
\wh Q^{\frac{1}{2}} - {Q^\ast}^{\frac{1}{2}} 
\geq \dotp{\partial_\beta \big(Q^{\frac{1}{2}}\big)(\beta^*,\theta^*) }{\wh\beta-\beta^*} + \dotp{\partial_\theta\big(Q^{\frac{1}{2}}\big)(\beta^*,\theta^*)}{\wh\theta-\theta^*},
\]
where $\partial_\beta\big(Q^{\frac{1}{2}}\big)(\beta^*,\theta^*)$ $\Big($ or $\partial_\theta\big(Q^{\frac{1}{2}} \big)(\beta^*,\theta^*)\Big)$ represents the subgradient of $Q^{\frac{1}{2}}$ with respect to $\beta$ (or $\theta$), evaluated at the point $(\beta^*, \theta^*)$.
If $Q^\ast\neq 0$, we have that
\[
\partial_\beta\big(Q^{\frac{1}{2}}\big)(\beta^*,\theta^*) = -\frac{1}{2} {Q^\ast}^{-\frac12} \cdot \frac{1}{n} X^T \l( Y-X\beta^*- \sqrt{n}\theta^* \r)
\]
and
\[
\partial_\theta \big(Q^{\frac{1}{2}}\big)(\beta^*,\theta^*) = -\frac{1}{2} {Q^\ast}^{-\frac12} \cdot \frac{1}{\sqrt{n}} \l( Y-X\beta^*- \sqrt{n}\theta^* \r).
\]
Therefore,
\begin{multline*}
\wh Q^{\frac{1}{2}} - {Q^\ast}^{\frac{1}{2}} \geq -\frac{\frac{1}{n}\sum_{j=1}^{n}(y_j-X_j^T\beta^*-\sqrt{n}\theta_j^*)X_j^T(\wh\beta-\beta^*) }{2{Q^\ast}^{\frac{1}{2}}} \\
- \frac{\frac{1}{\sqrt{n}}\sum_{j=1}^{n}(y_j-X_j^T\beta^*-\sqrt{n}\theta_j^*)(\wh\theta_j-\theta_j^*)}{2{Q^\ast}^{\frac{1}{2}}} \\ 
\geq -\sigma \frac{|\xi^\top X \Delta^\beta|/n }{2Q(\beta^*,\theta^*)^{\frac{1}{2}}} - \sigma \frac{\frac{1}{\sqrt{n}}\max_{j\geq o'}(|\xi|_{(j)}/\mu_j)\l\|\wh\theta-\theta^*\r\|_\mu}{2{Q^\ast}^{\frac{1}{2}}} \\
\geq -\sigma \frac{(\|\Delta^\beta\|_\lambda  + \|\Delta^\theta\|_\mu)/10 + 2\sqrt{\frac{\log(1/\delta)}{n}}\|\Delta^\beta\|_\Sigma }{2{Q^\ast}^{\frac{1}{2}}},
\end{multline*}
where the last inequality follows from \eqref{eq:test arxiv 1} combined with that fact that 
\[
\max_{j\geq o'}(|\xi|_{(j)}/(\sqrt{n}\mu_j)) \leq 1/10.
\] 
We conclude that
\begin{equation}
\label{lemma 4.1 proof eq2}
2{Q^\ast}^{\frac{1}{2}} \l( \wh Q^{\frac{1}{2}} - {Q^\ast}^{\frac{1}{2}}\r)/\sigma \geq -(\|\Delta^\beta\|_\lambda  + \|\Delta^\theta\|_\mu)/10 - 2 \sqrt{\frac{\log(1/\delta)}{n}}\|\Delta^\beta\|_\Sigma.
\end{equation}
Recall that $\|\xi\|_2^2 \geq n/10$ on event $\mathcal{E}$. 
Hence, in the view of the inequality \eqref{lemma 4.1 proof eq2}, we see that 
\begin{equation}
\label{Q_hat_Q_star_lower_bound}
 \l( \wh Q^{\frac{1}{2}} - {Q^\ast}^{\frac{1}{2}} \r) \geq -1/8\|\Delta^\beta\|_\lambda  - 1/8\|\Delta^\theta\|_\mu - 2\sqrt{\frac{\log(1/\delta)}{n}}\|\Delta^\beta\|_\Sigma.
\end{equation}
Combining the upper and the lower bounds above, we deduce the following result:
\[
\|\Delta^\beta\|_\lambda  + \|\Delta^\theta\|_\mu  \leq 4\sqrt{\frac{\sum_{i=1}^s \lambda_i^2}{\kappa(s)} + \frac{\log(1/\delta)}{n}} \|\Delta^\beta\|_\Sigma + 4\sqrt{\sum_{i=1}^{o'} \mu_i^2} \|\Delta^\theta\|_2 ,
\]
as claimed.
\end{proof}


We are ready to proceed with the proof of the main result. Since the loss function $L(\beta,\theta)$ defined in \eqref{equation:sqrt lasso loss} is convex, the first-order optimality conditions ensure the existence of $\wh{v}\in \partial\l\|\wh{\beta}\r\|_1$, $\wh{u}\in\partial\l\|\wh{\theta}\r\|_1$ such that $\wh{v}^T\wh\beta = \l\|\wh\beta\r\|_\lambda$, $\wh{u}^T\wh\theta = \l\|\wh\theta\r\|_\mu$, and 
\begin{equation*}
0 = -\frac{\frac{1}{n}X^T(Y-X\wh{\beta}-\sqrt{n}\wh{\theta})}{2\wh{Q}^{\frac12}} +  \wh{v},
\end{equation*}
\begin{equation*}
0 = -\frac{\frac{1}{\sqrt{n}}(Y-X\wh{\beta}-\sqrt{n}\wh{\theta})}{2\wh{Q}^{\frac12}} +  \wh{u}
\end{equation*}
under the assumption that $\wh{Q} \neq 0$. The two equations above are equivalent to 
\begin{equation}\label{eq:kkt condition for sqrt lasso}
[X^{(n)}, I_n]^T(Y^{(n)} - X^{(n)}\wh{\beta} -\wh{\theta}) = 2\wh{Q}^{\frac12} \l[  \l( \wh{v} \r)^T, \l( \wh{u}\r)^T \r]^T.
\end{equation}
Note that when $\wh{Q} = 0$, we have that $\wh{\xi} = 0$ and hence \eqref{eq:kkt condition for sqrt lasso} is still valid. Next, recall that $Y^{(n)} = X^{(n)}\wh{\beta} + \wh{\theta} + \sigma \xi^{(n)}$, so by \eqref{eq:kkt condition for sqrt lasso},
\begin{equation*}
[X^{(n)}, I_n]^T[X^{(n)}, I_n]\Delta = \sigma [X^{(n)}, I_n]^T\xi^{(n)} - 2\wh{Q}^{\frac12} \l[ \l( \wh{v} \r)^T, \l( \wh{u}\r)^T \r]^T.
\end{equation*}
Multiplying both sides of this equation by $\Delta^T$ from the left, we get
\begin{multline}\label{eq:kkt after multiplying}
\l\| X^{(n)} \Delta^\beta + \Delta^\theta \r\|_2^2 = \sigma (\Delta^\beta)^T (X^{(n)})^T\xi^{(n)} + \sigma (\Delta^\theta)^T \xi^{(n)} 
- 2\wh{Q}^{\frac12}(\Delta^\beta)^T\wh{v} - 2\wh{Q}^{\frac12}(\Delta^\theta)^T \wh{u}.
\end{multline}
Note that $ \underset{i}{\max} (|\wh{v}|_{(i)}/ \lambda_i) \leq 1$ and $\wh{v}^T\wh\beta = \l\|\wh\beta\r\|_\lambda$, hence
\[
-(\Delta^\beta)^T\wh{v} = (\beta^*-\wh\beta)^T\wh{v} = (\beta^*)^T\wh{v} - \l\|\wh\beta\r\|_1 \leq \l\|\beta^*\r\|_\lambda - \l\|\wh\beta\r\|_\lambda.
\]  
Similarly, we can show that
\[
-(\Delta^\theta)^T\wh{u} \leq \l\|\theta^*\r\|_\mu - \l\|\wh{\theta}\r\|_\mu.
\]
Combining these bounds with equation \eqref{eq:kkt after multiplying} and noticing that $(Q^*)^{\frac12} = \frac{\sigma}{\sqrt{2n}}\l\|\xi\r\|_2$, we deduce that on event $\mathcal{E}$,
\begin{multline}\label{ineq: F_1_bound_step_1}
\l\| X^{(n)} \Delta^\beta + \Delta^\theta \r\|_2^2 \leq \sigma(\|\Delta^\beta\|_\lambda  + \|\Delta^\theta\|_\mu)/2 + \sigma\sqrt{\frac{\log(1/\delta)}{n}}\|\Delta^\beta\|_\Sigma \\
 + 2\wh{Q}^{\frac12}  \l( \l\|\beta^*\r\|_\lambda - \l\|\wh\beta\r\|_\lambda \r) +  2\wh{Q}^{\frac12}   \l( \l\|\theta^*\r\|_\mu - \l\|\wh\theta\r\|_\mu \r).
\end{multline}
For brevity, set $x = \l\| X^{(n)} \Delta^\beta + \Delta^\theta \r\|_2$.

Note that 
\[
\hat{Q}^{1/2} = \frac{1}{\sqrt{2n}}\| Y - X\hat{\beta} - \sqrt{n} \hat{\theta} \|_2 = \frac{1}{\sqrt{2}}\| \xi^*/\sqrt{n} - X^{(n)}\Delta^\beta - \Delta^\theta \|_2.
\]
Since  $\l(Q^*\r)^{1/2} = \frac{1}{\sqrt{2n}}\| \xi^*  \|_2$, we deduce that
\[
\hat{Q}^{1/2} \leq \l(Q^*\r)^{1/2} + x/\sqrt{2}. 
\]
 Therefore, we derive using the inequality \eqref{ineq: F_1_bound_step_1} that
\ben{
\label{equation:thm1 middle result 0}
x^2 \leq  (5\sigma + x)(\|\Delta^\beta\|_\lambda  + \|\Delta^\theta\|_\mu) + \sigma\sqrt{\frac{\log(1/\delta)}{n}}\|\Delta^\beta\|_\Sigma,
}
where in the last step we used the bound $\l\|u\r\|_\lambda - \l\|v\r\|_\lambda \leq \l\|u-v\r\|_\lambda$ that holds for any vectors $u$ and $v$.
Lemma \ref{lemma:sparsity inequality} implies that  
\[
\|\Delta^\beta\|_\lambda  + \|\Delta^\theta\|_\mu  \leq 4\sqrt{\frac{\sum_{i=1}^s \lambda_i^2}{\kappa(s)} + \frac{\log(1/\delta)}{n}} \|\Delta^\beta\|_\Sigma + 4\sqrt{\sum_{i=1}^{o'} \mu_i^2} \|\Delta^\theta\|_2.
\]
Using Proposition \ref{prop:RE} and the condition on the sample size $n$, we have that
\[
x^2 \geq (\|\Delta^\beta\|^2_\Sigma + \|\Delta^\theta\|_2^2 )/16,
\]
and that
\[
\|\Delta^\beta\|_\lambda  + \|\Delta^\theta\|_\mu  \leq x/2.
\]
Combining these bounds with \eqref{equation:thm1 middle result 0}, we get
\begin{multline}
\label{equation:thm1 middle result 1}
(\|\Delta^\beta\|^2_\Sigma + \|\Delta^\theta\|_2^2 )/16\leq  x^2 \leq 10\sigma \l(\sqrt{\frac{\sum_{i=1}^s \lambda_i^2}{\kappa(s)} + \frac{\log(1/\delta)}{n}} \|\Delta^\beta\|_\Sigma + \sqrt{\sum_{i=1}^{o'} \mu_i^2} \|\Delta^\theta\|_2 \r).
\end{multline}
Therefore,
\[
\l\|\Delta^\beta\r\|^2_\Sigma + \l\|\Delta^\theta\r\|^2_2 \leq 100 \sigma^2\l(  \frac{\sum_{i=1}^s \lambda_i^2}{\kappa(s)} + \frac{\log(1/\delta)}{n} + \sum_{i=1}^{o'} \mu_i^2\r).
\]
Moreover,
\be{
\|\Delta^\beta\|_\lambda  + \|\Delta^\theta\|_\mu  \leq 100 \sigma\l(  \frac{\sum_{i=1}^s \lambda_i^2}{\kappa(s)} + \frac{\log(1/\delta)}{n} + \sum_{i=1}^{o'} \mu_i^2\r)
}
which concludes  the proof.

\begin{remark}\label{remark: difference between Q_true and Q_hat}
Direct computation shows that 
\ml{
\l| (Q^*)^{\frac12} - (\wh{Q})^{\frac12} \r| = \frac{1}{\sqrt{2}} \l| \l\|Y- X^{(n)}\beta^* - \theta^*\r\|_2 - \l\|Y-X^{(n)}\wh{\beta} - \wh{\theta}\r\|_2  \r|
\\
\leq \frac{1}{\sqrt{2}} \l\|X^{(n)} \Delta^\beta + \Delta^\theta\r\|_2,
}
from which we derive that
\[
\l| (Q^*)^{\frac12} - (\wh{Q})^{\frac12} \r|  \leq  (Q^*)^{\frac12}/20.
\]
 This shows that under the assumptions of the lemma, $\sqrt{\wh{Q}}$ can be close to $\sqrt{Q^*}$. This fact will be important in the next proof.
\end{remark}

Finally, we present the proof of the error bound for the estimator $\wh{\beta}$ only, as opposed to the vector $(\wh\beta,\wh\theta)$. The main idea of the proof, first used in \cite{dalalyan2019outlier}, is to treat $\theta^*$ as a ``nuisance parameter'' and repeat parts of the previous argument. Recall the key notation: 
\begin{itemize}
	\item $\wh{Q}:=Q(\wh{\beta},\wh{\theta})$ and $Q^*:= Q(\beta^*,\theta^*)$;
	\item $A^{(n)} := \frac{1}{\sqrt{n}} A$, whenever $A$ is a number, vector or matrix;
	\item $\wh{\xi} := Y - X \wh{\beta} - \sqrt{n}\wh{\theta}$;
	\item Also, we note that $\wh{Q} =  \frac{1}{2n} \l\| \wh{\xi} \r\|_2^2 $ and $Q^* = \frac{1}{2n} \l\|\xi \r\|_2^2$.
\end{itemize}
Since 
\[
\wh{\beta} \in \argmin_{\beta} \l\{ \frac{1}{\sqrt{2n}} \l\|Y-X\beta - \sqrt{n}\wh\theta\r\|_2 +  \l\|\beta\r\|_\lambda  \r\},
\]
there exists $\wh{v} \in \partial \l\|\wh\beta\r\|_1$ with $\wh{v}^T \wh{\beta} = \l\|\wh{\beta}\r\|_\lambda$ such that
\[
-\frac{1}{n}X^T (Y-X\wh\beta - \sqrt{n} \wh{\theta}) + 2 (\wh{Q})^{\frac12}  \wh{v} = 0.
\]
It implies, together with the identity $Y^{(n)} = X^{(n)}\wh{\beta} + \wh{\theta} +\sigma \xi^{(n)}$, that
\[
(X^{(n)})^T \l( X^{(n)}\Delta^\beta + \Delta^\theta - \sigma \xi^{(n)} \r) + 2(\wh{Q})^{\frac12}\lambda_s\wh{v} = 0.
\]
Multiplying both sides from the left by $\l( \Delta^\beta \r)^T$, we get that
\begin{equation}
\label{equation:fixed design improvement 0}
\l\|X^{(n)}\Delta^\beta\r\|_2^2 = - \dotp{X^{(n)}\Delta^\beta}{\Delta^\theta} + \dotp{X^{(n)}\Delta^\beta}{\sigma \xi^{(n)}} - 2(\wh{Q})^{\frac12}  \dotp{\Delta^\beta}{\wh{v}}.
\end{equation}
Recall that 
\[
-\dotp{\Delta^{\beta}}{\wh{v}} \leq \l\|\beta^*\r\|_\lambda - \l\|\wh\beta\r\|_\lambda \leq 2\sqrt{\frac{\sum_{i=1}^s \lambda_i^2}{\kappa(s)}} \|\Delta^\beta\|_\Sigma - \|\Delta^\beta\|_\lambda/2,
\]
which together with remark \ref{remark: difference between Q_true and Q_hat} implies that
\[
\frac{1}{2} \sigma \leq ( \wh{Q} )^{\frac12} \leq \frac{3}{2} \sigma.
\]
Moreover, from the inequality \eqref{eq:property3} we see that
\[
\dotp{X^{(n)}\Delta^\beta}{\xi^{(n)}} \leq  \|\Delta^\beta\|_\lambda/20  +  \sqrt{\frac{\log(1/\delta)}{n}}\|\Delta^\beta\|_\Sigma.
\]
Combining these results with the relation \eqref{equation:fixed design improvement 0}, we deduce that on the event $\mathcal{E}$,
\begin{multline*}
\l\|X^{(n)}\Delta^\beta\r\|_2^2 \leq - \dotp{X^{(n)}\Delta^\beta}{\Delta^\theta}  +\sigma\l(  \sqrt{\frac{\log(1/\delta)}{n}}\|\Delta^\beta\|_\Sigma + 2\sqrt{\frac{\sum_{i=1}^s \lambda_i^2}{\kappa(s)}} \|\Delta^\beta\|_\Sigma - \|\Delta^\beta\|_\lambda/2 \r).
\end{multline*} 
Inequality \eqref{eq:property2} yields that
\[
- \dotp{X^{(n)}\Delta^\beta}{\Delta^\theta}  \leq\|\Delta^\beta\|_\lambda \|\Delta^\theta\|_2/10 + \|\Delta^\theta\|_\lambda \|\Delta^\beta\|_\Sigma/10 + C \sqrt{\frac{\log(1/\delta)}{n}}\|\Delta^\beta\|_\Sigma\|\Delta^\theta\|_2.
\]
Applying \eqref{eq:property1}, we deduce that
\[
\l\|X^{(n)}\Delta^\beta\r\|_2^2\geq \frac{1}{2} \|\Delta^\beta\|_\Sigma^2 - \|\Delta^\beta\|^2_\lambda/4.
\]
Since $\|\Delta^\beta\|_\lambda + \|\Delta^\theta\|_2  \ll 1$,
\be{
\label{equation:fixed design improvement 1}
\frac{1}{2} \|\Delta^\beta\|_\Sigma^2 \leq \sigma\l(  \sqrt{\frac{\log(1/\delta)}{n}} + 2\sqrt{\frac{\sum_{i=1}^s \lambda_i^2}{\kappa(s)}} +  \|\Delta^\theta\|_\lambda /10  \r)\|\Delta^\beta\|_\Sigma.
}
Therefore,
\[
\|\Delta^\beta\|_\Sigma^2 \leq 100\sigma^2\l(  \frac{\log(1/\delta)}{n} + \frac{\sum_{i=1}^s \lambda_i^2}{\kappa(s)} +  \|\Delta^\theta\|^2_\lambda/2  \r).
\]
Observe that
\begin{align*}
\|\Delta^\theta\|^2_\lambda/2 &\leq \l( \sum_{j\geq o'} \lambda_j |\Delta^\theta|_{(j)} \r)^2 +  \|\Delta^\theta\|_2^2 \sum_{j=1}^{o'} \lambda_j^2 \\
&\leq \l( \sum_{j\geq o'} (\lambda_j/\mu_j) \mu_j  |\Delta^\theta|_{(j)} \r)^2 + 10 \sum_{j=1}^{o'} \lambda_j^2(1+ \sum_{j=1}^{o'} \mu_j^2)\\
&\leq \underset{j\geq o'}{\max}(\lambda^2_j/\mu^2_j)(\|\Delta^\theta\|^2_\mu+20(\sum_{j=1}^{o'} \mu_j^2)^2 ) + 10\sum_{j=1}^{o'} \lambda_j^2,
\end{align*}
where we used the inequality $\sum_{j=1}^{o'} \lambda_j^2 \leq 2 o'\lambda^2_{o'} \leq 2(\lambda^2_{o'}/\mu^2_{o'})\sum_{j=1}^{o'} \mu_j^2$. Since $\lambda_j \leq \mu_j$, we conclude that
\[
\|\Delta^\beta\|_\Sigma^2 \leq 100\sigma^2\l(  \frac{\log(1/\delta)}{n} + \frac{\sum_{i=1}^s \lambda_i^2}{\kappa(s)} + \underset{j\geq o'}{\max}(\lambda^2_j/\mu^2_j) (\sum_{j=1}^{o'} \mu_j^2)^2  \r).
\]
\subsection{Proof of Theorems \ref{thm:sqrt lasso prob bound} and  \ref{thm:sqrt lasso prob bound 2}. }

Recall that event $\mathcal{E}$ and the properties of sub-Gaussian designs expressed via the inequalities \eqref{eq:property1}, \eqref{eq:property2}, \eqref{eq:property3} hold for a given $\delta$ with probability at least $1-3\delta - 5\exp(-co')$. Observe that $\sum_{i=1}^s \lambda_i^2 \leq Cs\log(ep/s)/n$. For the case of fixed thresholds $\mu_i = \frac{C}{\sqrt{n}} \l(\frac{n}{m}\r)^{1/\tau}$, where $m = \log(1/\delta)$, we get that
\begin{multline*}
\underset{j\geq o'}{\max}\l(\lambda^2_j/\mu^2_j \r) \l(\sum_{j=1}^{o'} \mu_j^2 \r)^2 = C \l(\frac{o'}{n}\r)^2\log(n/o')\l(\frac{n}{m}\r)^{2/\tau}
\\
=C\l(\frac{o'}{n}\r)^{2-2/\tau}\log(n/o')\l(\frac{o'}{m}\r)^{2/\tau}.
\end{multline*}
Choosing $o' = o + m$ with $m = \log(1/\delta)$, we deduce that
\[
\underset{j\geq o'}{\max}(\lambda^2_j/\mu^2_j) (\sum_{j=1}^{o'} \mu_j^2)^2 \leq C\l(\l(\frac{o}{n}\r)^{2-2/\tau}\log(n/o)(1+(o/m)^{2/\tau})  +\frac{\log(1/\delta)}{n} \r).
\]
Theorem \ref{thm: sqrt lasso bound} immediately yields that with probability at least $1-8\delta$,
\[
\|\Delta^\beta\|_\Sigma^2 \leq C\sigma^2\l(  \frac{\log(1/\delta)}{n} + \frac{s\log(ep/s)/n}{\kappa(s)} + \l(\frac{o}{n}\r)^{2-2/\tau}\log(n/o)(1+(o/\log(1/\delta))^{2/\tau})  \r).
\]
This completes the proof of Theorem \ref{thm:sqrt lasso prob bound}.

For the case of the adaptive threshold $\mu_i = \frac{C}{\sqrt{n}} \l(\frac{n}{i}\r)^{1/\tau}$, for any $\delta$ we can choose $o' = o + m$ with $m = \log(1/\delta)$ so that
\[
\underset{j\geq o'}{\max}\l(\lambda^2_j/\mu^2_j\r) \l(\sum_{j=1}^{o'} \mu_j^2\r)^2 \leq C\l( \frac{\sum_{i=1}^{o'} (n/i)^{2/\tau}}{n}\r)^2\leq C \frac{(\sum_{i=1}^{o'} (1/i)^{2/\tau})^2}{n^{2-4/\tau}}.
\]
Hence, for $\tau>2$ we have 
\[
\underset{j\geq o'}{\max}(\lambda^2_j/\mu^2_j) (\sum_{j=1}^{o'} \mu_j^2)^2 \leq C(o'/n)^{2-4/\tau}.
\]
In view of Theorem \ref{thm: sqrt lasso bound},
\[
\|\Delta^\beta\|_\Sigma^2 \leq C\sigma^2\l(  \frac{\log(1/\delta)}{n} + \frac{s\log(ep/s)/n}{\kappa(s)} + (o + \log(1/\delta)/n)^{2-4/\tau}\r)
\]
with probability at least $1-8\delta$. 
Finally, if the noise is sub-Gaussian, then
\[
\underset{j\geq o'}{\max}(\lambda^2_j/\mu^2_j) (\sum_{j=1}^{o'} \mu_j^2)^2 \leq C \l(\frac{o'\log(n/o')}{n} \r)^2,
\]
and we get that
\[
\|\Delta^\beta\|_\Sigma^2 \leq C\sigma^2\l(  \frac{\log(1/\delta)}{n} + \frac{s\log(ep/s)/n}{\kappa(s)} +\l(\frac{o\log(n/o)}{n} \r)^2\r),
\]
again with probability at least $1-8\delta$. 
The last inequality holds since 
\[
\frac{\log(1/\delta)}{n} \geq \l(\frac{\log(1/\delta)\log(n/\log(1/\delta))}{n}\r)^2
\] 
whenever $\log(1/\delta)$ is smaller than $n$. This completes the proof of Theorem \ref{thm:sqrt lasso prob bound 2}.